\documentclass[leqno,12pt,draft]{article}

\usepackage{tikz}
\usepackage[a4paper]{geometry}
\usepackage{graphicx}
\usepackage{amsfonts,amssymb,amsthm}
\usepackage[centertags]{amsmath}
\allowdisplaybreaks[1]
\usepackage{caption}
\usepackage{paralist}
\usepackage{subfigure}
\usepackage{amsthm} 
\newcounter{theorems}

\theoremstyle{plain}

\newtheoremstyle{par}
     {\topsep}
     {\topsep}
     {\itshape}
     {}
     {\bfseries}
     {}
     {.5em}
     {}

\newtheoremstyle{parrm}
     {\topsep}
     {\topsep}
     {\normalfont}
     {}
     {\itshape}
     {}
     {.5em}
     {}

\theoremstyle{par}
\theoremstyle{parrm}
\numberwithin{equation}{section}
\swapnumbers
\makeatletter
\def\tagform@#1{\maketag@@@{\ignorespaces#1\unskip\@@italiccorr}}

\makeatother
\usepackage[iso,nocleanlook,english]{isodate}
\usepackage[autostyle]{csquotes}
\usepackage[mathstyleoff]{breqn} 
\usepackage[all,pdf,curve,cmtip]{xy}
\newdir{ >}{!/10pt/@{ }*@{>}}
\newdir{< }{!/10pt/@{ }*@{<}}
\newcommand{\term}[1]{\emph{#1}}

\newcommand{\RR}{\mathbb{R}}

\newcommand{\from}{\colon}

\usepackage{bm}
\newcommand{\simbolovettore}[1]{{\boldsymbol{#1}}}

\newcommand{\vm}{\simbolovettore{m}}
\newcommand{\vn}{\simbolovettore{n}}

\newcommand{\vq}{\simbolovettore{q}}

\newcommand{\vu}{\simbolovettore{u}}
\newcommand{\vv}{\simbolovettore{v}}

\newcommand{\zero}{\simbolovettore{0}}

\newcommand{\abs}[1]{\lvert{#1}\rvert}

\newcommand{\tT}{\mkern-1.5mu\mathsf{T}}

\newcommand{\pfaff}{\operatorname{Pf}}

\usepackage{enumitem}
\providecommand{\varitem}{} 
\makeatletter
\newenvironment{axioms}[1]
 {\renewcommand\varitem[1]{\item[\textbf{#1\arabic{enumi}\rlap{$##1$}:}]%
    \edef\@currentlabel{#1\arabic{enumi}{$##1$}}}%
  \enumerate[nolistsep,label=\textbf{#1\arabic*:}, ref=#1\arabic*, start=1]}
 {\endenumerate}
\makeatother

\newcommand{\aref}[1]{\textbf{\textup{\ref{#1}}}}

\newcommand{\conf}[2]{\mathbb{F}_{#1}(#2)}
\usepackage[english]{babel}
\theoremstyle{plain}
\newtheorem{lemma}[equation]{Lemma}
\newtheorem{theo}[equation]{Theorem}
\newtheorem{coro}[equation]{Corollary}
\newtheorem{propo}[equation]{Proposition}

\theoremstyle{par}
\theoremstyle{definition}
\newtheorem{defi}[equation]{Definition}
\newtheorem{remark}[equation]{Remark}
\allowdisplaybreaks

\usepackage[natbib=false,backend=bibtex,style=numeric,doi=false,isbn=false,url=false]{biblatex}
\title{On Certain Pfaffians Connected with the Inverse Problem for Collinear Central Configurations}
\author{DL~Ferrario}
\date{\isodate\today}

\newcommand{\quadmat}{\mathcal{M}}
\newcommand{\matrices}{\quadmat}
\newcommand{\bmatrices}{\mathcal{B}}
\newcommand{\De}[2][x]{\dfrac{\partial #2}{\partial #1}}

\begin{document}
\selectlanguage{english}

\maketitle

\begin{abstract}
A. Albouy and R. Moeckel in 2000 found some interesting inequalities related to 
the inverse problem for
collinear (Moulton) central configurations: the Pfaffian 
of a certain matrix is positive since all coefficients of some polynomials 
are positive, for the Newtonian (interaction potential $1/r$ and even $n\leq 6$). 
They conjectured that for all even $n$ such Pfaffians, for the Newtonian case, are positive. 
In this article we analyze further the problem, and we prove that these Pfaffians are positive 
for every even $n\leq 14$, for a broader class of interaction potentials ($\varphi$):
namely, we assume that the function $V=-\varphi'$ is positive, decreasing, and log-convex.  
This class includes all homogeneous potentials ($\varphi(t)=|t|^{-\alpha}$)
of negative homogeneity degree ($-\alpha$, with $\alpha>0$), and their finite sums. 
In doing so, the main tool is 
a family $p_n(u)$ of Pfaffian polynomials, 
that (for all even $n\leq 14$) happen to have only non-negative coefficients, and
with the property that $p_n(u)>0$. 
This yields an explicit positive lower bound for the relevant Pfaffians and, in
particular, the nonsingularity of all the corresponding matrices.


\end{abstract}

\section{Introduction}
\label{section:intro}

In \cite{AlbouyInverseProblemCollinear2000}, A. Albouy and R. Moeckel considered the 
inverse problem for collinear central configurations: given a collinear configuration 
of $n$ point masses, find the value of the masses making it central. 
In the study of the resulting Pfaffians, a remarkable property was found: each term in the 
polynomials occurring in the rational function expressing the Pfaffian for even $n\leq 6$ was 
positive. Hence the Pfaffians were proven to be positive. The authors conjectured that 
for the Newtonian potential and all even $n$, all resulting Pfaffians are nonzero, for all non-singular configurations 
(this is now known as the Albouy-Moeckel conjecture).
As observed at the end of \cite{AlbouyInverseProblemCollinear2000}, 
A. Chenciner pointed out the old 
paper \cite{Buchanancertaindeterminantsconnected1909} of H.E. Buchanan, 
which claims to prove their conjecture, with an incorrect (and probably unrecoverable) argument.
Since Albouy-Moeckel proof for $n=6$ and $\alpha=1$ was computer-assisted,  
Xie in  \cite{xieAnalyticalProofCertain2014a}
established an analytical proof
(see also \cite{xieRemarksInverseProblem2022b} for a survey of related results). 
In \cite{ferrarioPfaffiansInverseProblem2020c} and 
\cite{ferrarioMultivaluedFixedPoints2020a} I tried to understand better the reason 
of the positivity of the Pfaffians, and to improve the results about the conjecture; 
I could prove
that the resulting Pfaffians are positive for any homogeneity $\alpha >0$ 
and even $n\leq 6$, or for $\alpha=1$ and even $n \leq 10$ with a computer-assisted proof.

Such results can be improved in two directions: by considering more 
general interacting potentials, and by increasing the value of $n$ (even). 
In order to do so, two main ingredients have been introduced. The first, is 
the candidate space of some anti-symmetric matrices, constrained to 
some quadratic inequalities. The space,
defined below in \ref{defi:T}, is meant to have the property that all its elements 
have positive Pfaffian. 
In order to prove it, we introduce the second main ingredient: a suitable
parametrization 
 of the space of matrices, defined in \ref{defi:vij} and \ref{defi:uij}. 
 In terms of some of these parameters, we again find the remarkable fact that the Pfaffian 
 is a polynomial with positive coefficients.
 So, after an introductory section \ref{section:conj}, 
 in section \ref{section:2} we introduce 
the space $\matrices_n$ of all $n\times n$ anti-symmetric matrices with positive upper-right entries,
satisfying some simple inequalities on their coefficients. 
Such matrices can be manipulated and simplified, eliminating $n$ entries. The space
 $\bmatrices_n$ of all such simplified matrices 
 can be parametrized \ref{defi:vij} with $\frac{n(n-3)}{2}$ 
variables $v_{ij}$, as described in proposition \ref{propo:vij}. In terms of the new variables,
the inequalities defining $\matrices_n$ and $\bmatrices_n$ can be re-written
simply as $v_{ij} \geq 1$ for all $i,j$. 
By letting $v_{ij} = 1 + u_{ij}$, the Pfaffian of a generic matrix in $\bmatrices_n$
is a polynomial in the variables 
$u_{ij}$: even if it is an alternating sum of terms, after simplifications what 
turns out to be true is that all its non-zero coefficients are positive. 
The proof for $n\leq 6$ is simple. In order to prove it for (even) $n\geq 8$, a computer-assisted 
proof has been necessary. This is the content of section \ref{section:3}, where 
the algorithm is described. Here the idea is basically that instead of computing all the coefficients 
of the polynomials in $u_{ij}$-variables, it is possible to eliminate the $v_{ij}$-variables one after the other, 
using some suitable technical lemmas, without actually computing all the
coefficients of the polynomial in $u_{ij}$. 
The mathematical and computing problems are highlighted in section \ref{section:3}. 
Finally, in section \ref{section:5} we apply the previous results to
central configurations and the Albouy-Moeckel conjecture:
the goal is to show that matrices occurring from the Newtonian potential, or a homogeneous potential
of homogeneity degree $-\alpha$, are always elements of the space of matrices $\matrices_n$ 
described above. This turns out to be easily proven (Theorem \ref{theo:pfaffQ}). 
But, more interestingly, a much wider class of potentials have the property that the corresponding
matrix belongs to $\matrices_n$: it suffices that the entries of the matrix can be written as 
$Q_{ij} = V(q_i -q_j)$, 
where $V$ is  a function which is monotone decreasing
and log-convex (here $(q_1,\ldots, q_n)\in \RR^n$ is a non-singular collinear configuration).  
This is proven in \ref{theo:2}. 
From these theorems 
one can deduce Corollaries \ref{coro:main} and \ref{coro:alphahomo},
that gives explicit estimates of the value of the (positive) Pfaffians.

\section{Albouy-Moeckel Conjecture and the Inverse Problem for Collinear Central Configurations}
\label{section:conj}

Let $n\geq 2$. Consider the configuration space of $n$ points in $\RR$,  denoted as
\( \conf{n}{\RR} = \{ \vq \in \RR^n : q_i \neq q_j \} \)
where $\vq = (q_1,q_2,\ldots, q_n) \in \RR^n$ and $\forall j$, $q_j \in \RR$. 
Given 
$n$ positive 
mass parameters $m_1,\ldots, m_n$,  
the potential function $U\from \conf{n}{\RR} \to \RR$ is defined as 
\begin{equation}
\label{eq:potential}
U(\vq)  
= 
\sum_{1\leq i \leq j \leq n} {m_i m_j} \varphi(q_i - q_j),
\end{equation}
where  $\varphi\from \RR \to \RR$ is a given even function.
If 
$\varphi(t) = \abs{t}^{-\alpha}$ for $\alpha >0$, 
then $U$ is the $\alpha$-homogeneous potential 
\begin{equation}
\label{eq:potentialalpha}
U(\vq) =  \sum_{1\leq i \leq j \leq n} \frac{m_i m_j}{\abs{q_i - q_j}^\alpha},
\end{equation}
which is the Newtonian gravitational potential for $\alpha=1$. 

Let $q_0 = \dfrac{\sum_{j=1}^n m_j q_j}{\sum_{j=1}^n m_j }$ be the center of mass of $\vq$. 
A \term{central configuration} for $U$ 
is a configuration $\vq \in \conf{n}{\RR}$ 
which is  a solution of the equation (in $(\vq,\lambda)$)
\begin{equation}\label{eq:CC0}
\De[q_j]{U} 
=
\sum_{\substack{k=1\\k\neq j} }^n m_j m_k \varphi'(q_j - q_k)
= \lambda m_j (q_j - q_0) \quad(j=1\ldots n),
\end{equation}
which in the $\alpha$-homogeneous case is 
\begin{equation*}
 \alpha \sum_{\substack{k=1\\k\neq j}}^n m_j m_k \dfrac{q_j - q_k}{\abs{q_j - q_k}^{\alpha+2}}  + 
\lambda m_j (q_j - q_0)  = 0 \quad (j=1\ldots n)~.
\end{equation*}
By setting
\begin{equation}
\label{eq:defQ}
Q_{ij} = 
-\varphi'(q_i - q_j)
=
\alpha \frac{q_i - q_j}{\abs{q_i - q_{j}}^{\alpha+2} }
\end{equation}
and eliminating some $m_j$, equation \ref{eq:CC0} can be written as
\begin{equation}\label{eq:CC2}
 \sum_{k\neq j}  Q_{jk} m_k   + 
\lambda (q_j -q_0)  = 0 \quad (j=1\ldots n),
\end{equation}
which can be written in matrix form as $Q\vm + \lambda (\vq - q_0 \boldsymbol{1}) = \zero$, 
where $Q$ is the skew-symmetric matrix with entries $Q_{jk}$, $\vm$ the vector
with components $m_i$ and $\boldsymbol{1}$ the vector with all components equal to $1$. 
Note that $V(t) = -\varphi'(t)$ is an odd function. 

By \cite[Lemma 1.5]{ferrarioPfaffiansInverseProblem2020c}, 
equation \eqref{eq:CC2} can be written as $n$ equations in the 
$n+1$ variables $(\hat M, \hat c)$, in matrix form
$Q\hat m + \hat c \boldsymbol{1} = \vq$. 
The \emph{inverse problem}—originally formulated by Moulton
\cite{Moultonstraightlinesolutions1910} and Buchanan
\cite{Buchanancertaindeterminantsconnected1909}, and later advanced by Albouy 
and Moeckel \cite{AlbouyInverseProblemCollinear2000}—reverses the standard
central configuration question. Specifically, it asks: if we are given a fixed
configuration $\vq$ (or their mutual distances
$q_i-q_j$), can one determine a set of strictly positive masses $m_j$ and a
multiplier $\lambda < 0$ that satisfy Equation \eqref{eq:CC2}?
It is clear that if the matrix $Q$ is invertible, one can 
derive $\vm = - \lambda Q^{-1}(\vq - q_0 \boldsymbol{1})$
(or, $\hat m = Q^{-1}(\vq - \hat c \boldsymbol{1})$), and it is left to check the 
positivity of masses. 

In his 1909 paper \cite{Buchanancertaindeterminantsconnected1909}, Buchanan
established a result that can be restated in modern terms as follows: for the
Newtonian potential ($\alpha=1$) and any even number of masses $n$, the
determinant of $Q$ is strictly non-vanishing, meaning $\det Q \neq 0$ for all
collinear configurations $\vq \in \conf{n}{\RR}$.

However, as Albouy and Moeckel revealed in
\cite{AlbouyInverseProblemCollinear2000}, Buchanan's original proof contains an
unrepairable error. Believing the underlying premise to still be
valid, they formalized the \emph{Albouy--Moeckel Conjecture}, which asserts
that the determinant is strictly non-zero for all configurations. Since then,
partial progress toward a complete resolution has been made. Albouy and Moeckel
\cite{AlbouyInverseProblemCollinear2000} provided proofs for
$n \leq 4$ with any $\alpha > 0$, as well as a computer-assisted proof for $n = 6$ in the Newtonian
case ($\alpha = 1$). 
Quite interestingly, and remarkably, 
they found that all polynomials involved in the rational expression of the Pfaffian 
have non-negative sign. Quoting the article, ``\emph{the numerator is a
polynomial in five positive variables with 15158 coefficients, all positive!}'' 
Subsequently, Xie \cite{xieAnalyticalProofCertain2014a}
established an analytical proof for the $n = 6$ case with $\alpha =
1$.
In
\cite{ferrarioPfaffiansInverseProblem2020c}, the author gave a proof for any $\alpha>0$ and $n=6$, 
and a computer-assisted proof for any even $n\leq 10$ in the Newtonian case
$\alpha=1$. 
Again, the computer-assisted proof was a more systematic and
efficient way of checking 
that all terms of polynomials involved in the rational expression of the Pfaffian 
do not have any negative term. 
We will see in section \ref{section:5} below that if the function $V$ 
satisfies the following two 
properties:
$V$ is decreasing
and log-convex, then the matrix 
$Q$ with entries $Q_{ij} = V(q_i - q_j)$ has non-zero Pfaffian $\pfaff Q$, for even $n\leq 14$.  


Recall the elementary fact
that if $n$ is odd and $A$ is an anti-symmetric
$n\times n$ matrix $A^{\tT} = - A \implies \det(A) = \det(-A) = (-1)^n \det(A) \implies \det(A) = 0$.
Consider the equation in matrix form $Q\hat \vm + \hat c \boldsymbol{1} = \vq$,
which can be written also as 
$\begin{bmatrix} 
Q & \boldsymbol{1} 
\end{bmatrix}
\begin{bmatrix}\hat\vm \\ \hat c \end{bmatrix} = \vq$: 
if $n$ is odd, a solution in $(\hat \vm, \hat c)$ exists
  if the $(n+1)\times (n+1)$ skew-symmetric 
matrix $Q^\mathrm{b}$ defined as
\[
Q^\mathrm{b} = \begin{bmatrix} 
Q & \boldsymbol{1} \\
-\boldsymbol{1}^{\tT} & 0 \\
\end{bmatrix}
\]
is non-singular. 
For the case of odd $n$, formulated in terms of the Pfaffian of the associated bordered matrix, 
see \cite{AlbouyInverseProblemCollinear2000},
\cite{xieAnalyticalProofCertain2014a}
and
\cite{xieRemarksInverseProblem2022b}.

If $n$ is even, 
the skew-symmetric determinant of even order is equal to the square of its Pfaffian:
if $A$ is a skew-symmetric $n\times n$ matrix, with $n$ even,
$\det A  = \left( 
\pfaff A
\right)^2$. 
Using Moulton's 1910 notation, the Pfaffian is defined as
\setlength{\arrayrulewidth}{1pt}
\[
\pfaff Q = 
\begin{array}{cccc|}
\multicolumn{1}{|c}{Q_{12}} & Q_{13} & \cdots &  Q_{1n}\\
 & Q_{23} & \cdots & Q_{2n} \\
 & & \ddots & \vdots \\
 && & Q_{n-1,n}
\end{array}
=
\sum_{\sigma} (-1)^\sigma  Q_{r_1,s_1} Q_{r_2,s_2} \ldots Q_{r_k,s_k}
\]
where $n=2k$, and the permutation
$\sigma$ runs over all \emph{perfect matchings} of $\vn=\{1,\ldots, n=2k\}$:
a perfect matching $\sigma$ is a fix-point free involution of $\vn$,
which can be represented also as
a partition of $\vn$ in pairs $[r_1,s_1, r_2,s_2,\ldots, r_k,s_k]$.
The sign $(-1)^\sigma$ is the parity of this permutation.
In D. Knuth and Cayley notation
\cite{knuthOverlappingPfaffians1996,cayley1849determinants}
\( \pfaff A = A[1,2,\ldots, n]  \).

The following recursive identity is the analogue of the Laplace expansion for the determinant:
\begin{equation}\label{eq:pfaff1}
A[1,2,\ldots,n]=\sum_{j=1}^{n-1} (-1)^{j+1} A_{jn} 
 A[1,\ldots, \hat j, \ldots, \hat n],
\end{equation}
where $A[1,\ldots, \hat j, \ldots, \hat n]$ denotes the Pfaffian of the matrix
 with the $j$-th and $n$-th rows and columns canceled out.

For $n=4$,
\[\begin{aligned}
\pfaff A & =  
A[1,2,3,4]   = A_{14} A[\hat 1,2,3,\hat 4] - A_{24} A[1,\hat 2,3,\hat 4] + A_{34} A[1,2,\hat 3,\hat 4]
\\
& = A_{14} A_{23} - A_{24} A_{13} + A_{34} A_{12}.
\end{aligned}
\]
In Moulton notation, again denoting with a hat all canceled symbols,
\[
\begin{aligned}
 \pfaff A   & = 
\setlength{\arrayrulewidth}{1pt}
\begin{array}{ccc|}
\multicolumn{1}{|c}{A_{12}} & A_{13} & A_{14}\\
 & A_{23} & A_{24} \\
 && A_{34}
\end{array}  
\\
 & =  A_{14}
\begin{array}{ccc|}
\multicolumn{1}{|c}{\hat A_{12}} & \hat A_{13} & \hat A_{14}\\
 & A_{23} & \hat A_{24} \\
 && \hat A_{34}
\end{array}
-
A_{24}  
\begin{array}{ccc|}
\multicolumn{1}{|c}{\hat A_{12}} & A_{13} & \hat A_{14}\\
 & \hat A_{23} & \hat A_{24} \\
 && \hat A_{34}
\end{array}
+ A_{34}   
\begin{array}{ccc|}
\multicolumn{1}{|c}{A_{12}} & \hat A_{13} & \hat A_{14}\\
 & \hat A_{23} & \hat A_{24} \\
 && \hat A_{34}
\end{array}  
.
\end{aligned}
\]

Some elementary properties of Pfaffians are the following: if $A$ is a skew-symmetric matrix,
and $B$ the matrix obtained by swapping the $i$-th and $j$-th columns and the $i$-th
and $j$-th rows, then
$\pfaff A = - \pfaff B$. 
Moreover, if $B$ is obtained by adding a multiple of the $j$-th column to the $k$-th column and 
by adding the same multiple of the $j$-th row to the $k$-th row of a skew-symmetric matrix $A$, then $\pfaff A = \pfaff B$. 
And, if $B$ is obtained by multiplying by a scalar $c$ both the $j$-th row and the $j$-th column of $A$, 
then $\pfaff B = c \pfaff A$.

\section{Pfaffian of Antisymmetric Matrices Subject to Quadratic and Monotonic Inequalities}
\label{section:2} 

In this section a particular type of matrices is introduced, 
which will be the object of our analysis. 
We are going to show that the 
Pfaffians of such matrices are positive, 
at least for $n\leq 14$. At the end of the article we will relate such 
matrices to the matrices $Q$, defined above in \eqref{eq:defQ}, 
of the Albouy-Moeckel conjecture.   

\begin{defi}
\label{defi:T}
Let $\quadmat_n$ be the space of all $n\times n$ antisymmetric 
matrices $A$ 
with real coefficients $a_{ij}$ 
satisfying the following properties: 
\begin{axioms}{A}
\item 
\label{positive}
\text{(positive)} 
$1\leq i<j \leq n \implies a_{ij} > 0$;
\item 
\label{quadratic}
\text{(quadratic inequality)} 
$a_{i,j}a_{i-1,j+1} \geq a_{i-1,j} a_{i,j+1} $ 
for all $i = 2,\ldots, n-2$, $ j = i+1 \ldots n-1$.  
 That is,
 $\displaystyle \det \begin{bmatrix} a_{i-1,j} & a_{i-1,j+1} \\
 a_{i,j} & a_{i,j+1} 
 \end{bmatrix} \leq 0 $ for all the $2\times 2$ submatrices in the upper (positive) triangular half of $A$. 
 \item
\label{monotone} (monotone)
For $j=2,\ldots n-2$ 
\begin{align*}
a_{1,2} \geq \ldots \geq a_{1,j} \geq a_{1,j+1} \geq \ldots \geq a_{1,n-1} \\
a_{2,n} \leq \ldots \leq a_{j,n} \leq a_{j+1,n} \leq \ldots \leq a_{n-1,n} 
\end{align*}
\end{axioms}
\end{defi}

The Pfaffian of a $4\times 4$ matrix in $\matrices_4$ is positive, as proven in the 
following proposition.  

\begin{propo}
\label{propo:4x4}
Let $Q$ be a $4\times 4$  matrix defined by
\[
Q = 
\begin{bmatrix}
0 & Q_{12} & Q_{13} & Q_{14}\\- Q_{12} & 0 & Q_{23} &
Q_{24}\\- Q_{13} & - Q_{23} & 0 & Q_{34}\\- Q_{14} & - Q_{24} & - Q_{34} ~.
& 0\end{bmatrix}
\]
Its Pfaffian $\pfaff{Q}$  is equal to 
$ Q_{12}Q_{34} - Q_{13}Q_{24} + Q_{14}Q_{23} $,
so that if $Q$ 
satisfies \aref{positive} and \aref{quadratic}, 
then $\pfaff{Q} \geq Q_{12}Q_{34} >0$.
On the other hand, if $Q$ satisfies \aref{positive} and \aref{monotone}, 
then $\pfaff{Q} \geq  Q_{14}Q_{23} > 0$. 
\end{propo}
\begin{proof}
If $Q$ satisfies \aref{positive} of \ref{defi:T}, 
then
$Q_{12}>0$, $Q_{13}>0$, $Q_{14}>0$, $Q_{23}>0$, $Q_{24}>0$, $Q_{34}>0$. 
By \aref{quadratic} of \ref{defi:T}, 
the only inequality that holds, for  $i=2$ and $j=3$, 
is 
\begin{align*}
\det 
\begin{bmatrix} 
Q_{13} & Q_{14} \\
Q_{23} & Q_{24} 
\end{bmatrix}  
& = Q_{13}Q_{24} - Q_{23}Q_{14} \leq 0 \\
\implies \pfaff{Q} & = 
Q_{12}Q_{34} - Q_{13}Q_{24} + Q_{14}Q_{23}  \geq Q_{12}Q_{34} > 0 . 
\end{align*}
Furthermore, 
if $Q$ satisfies \aref{positive} and \aref{monotone}, 
  then 
\begin{align*}
Q_{12} \geq Q_{13}, \quad 
Q_{24} \leq Q_{34}  
\implies 
Q_{12}Q_{34} \geq Q_{13}Q_{34} \geq  Q_{13}Q_{24},
\end{align*}
and hence 
\[
\pfaff Q = 
Q_{12}Q_{34} - Q_{13}Q_{24} + Q_{14}Q_{23}  \geq Q_{14}Q_{23} > 0 . 
\]
\end{proof}

\begin{remark}
\label{remark:counterexamplen6}
Thus, for $n=4$, any of the two properties \aref{quadratic} and \aref{monotone} (together with \aref{positive})
is sufficient to assert the positivity of the Pfaffian.  
Consider the following matrix:
\[
\begin{bmatrix}
0 & 0.0001 & 0.01 & 0.1 & 1 & 1\\
-0.0001 & 0 & 0.013 & 0.12 & 1.1 & 1\\
-0.01 & -0.013 & 0 & 0.22 & 2 & 1\\
-0.1 & -0.12 & -0.22 & 0 & 3 & 1\\
-1 & -1.1 & -2 & -3 & 0 & 1\\
-1 & -1 & -1 & -1 & -1 & 0
\end{bmatrix}
=
\begin{bmatrix}
0 & \frac{1}{10000} & \frac{1}{100} & \frac{1}{10} & 1 & 1\\
- \frac{1}{10000} & 0 & \frac{13}{1000} & \frac{3}{25} & \frac{11}{10} & 1\\- \frac{1}{100} & - \frac{13}{1000} & 0 & \frac{11}{50} & 2 & 1\\
- \frac{1}{10} & - \frac{3}{25} & - \frac{11}{50} & 0 & 3 & 1\\
-1 & - \frac{11}{10} & -2 & -3 & 0 & 1\\
-1 & -1 & -1 & -1 & -1 & 0
\end{bmatrix}
\]
Its Pfaffian is negative; properties \aref{positive} and \aref{quadratic} both hold true, and \aref{monotone} is false. 

Furthermore, consider the following matrix
\[
\begin{bmatrix}
0 & 4 & 3 & 2 & 1 & 1 \\
-4 & 0 & \frac{1}{2} & 1 & 1 & 1 \\
-3 & - \frac{1}{2} & 0 & \frac{1}{4} & 1 & 1 \\
-2 & -1 & -\frac{1}{4} & 0 & 1 & 1 \\
-1 & -1 & -1 & -1 & 0 & 1 \\
-1 & -1 & -1 & -1 & -1 & 0 \\
\end{bmatrix}
\]
Its Pfaffian is equal to $-1<0$; properties \aref{positive} and \aref{monotone} hold, while \aref{quadratic} does not. 
Finally, 

A third example of a $4\times 4$ matrix with \aref{monotone} and \aref{quadratic} holding, and without \aref{positive}, 
with negative Pfaffian is the following:
\[
\begin{bmatrix}
0 & \frac{1}{2} & -1 & -1 \\
-\frac{1}{2} & 0 & -1 & -1 \\
 1 & 1   & 0 & -1 \\
1 & 1 & 1 & 0 
\end{bmatrix}
\]

These examples illustrate two basic facts, that will be further clarified below: first, all properties
\aref{positive}, \aref{quadratic} and \aref{monotone} are necessary in the definition of the 
space of matrices $\matrices$, in order to have positive Pfaffians;
and second, an induction proof based on condensation identities (such as a  Desnanot-Jacobi formula 
for Pfaffians) and property of submatrices might be technically quite challenging,
 since all three properties need to be preserved in the induction process  (cf. remark \ref{remark:otherUL}).
\end{remark}

Now, before considering bigger matrices, 
a simpler  
 type of matrices in $\matrices_n$ is introduced, which keeps 
 the positivity of 
the Pfaffian. 

\begin{lemma}
\label{lemma:multiplyrowcol}
If $A$ is a matrix in $\matrices_n$,  and $k$ an index $1\leq k \leq n$,
then  for each 
positive constant $c>0$ and $\hat A$,
the matrix obtained by multiplying 
the $k$-th row and the $k$-th column of $A$ by $c$ is anti-symmetric, it 
satisfies 
properties \aref{positive} and \aref{quadratic},
and $\pfaff \hat A = c \pfaff A$.  
\end{lemma}
\begin{proof}
The matrix $\pfaff A$ is clearly anti-symmetric, and \aref{positive} holds. 
So see that \aref{quadratic} holds, it suffices to note that each square upper-right $2\times 2$ 
submatrix of $ \hat A$  is equal to the same sub-matrix of $A$ with either a column
or a row multiplied by $c$ (or none). 
\end{proof}

\begin{defi}
\label{defi:one-bordered}
A matrix  $B\in \matrices_n$ is termed \term{one-bordered}  if, 
for $j=1,\ldots, n-1$, $b_{j,n} = 1 $ and
  $b_{1,n-1} = 1$. 
The space of all one-bordered matrices in $\matrices_n$ is denoted by
$\bmatrices_n \subset \matrices_n$. 
\end{defi}

\begin{propo}
\label{propo:factors}
If $A$ is in $\matrices_n$, $n$ even, and
let $B$ be obtained by dividing the $j$-th row and column of $A$ by $a_{jn}$ for $j=1,\ldots, n-1$,
and then by dividing the resulting entries $a'_{ij} = \dfrac{a_{ij}}{a_{in}a_{jn}}$ 
for $i,j\leq n-1$ by $a'_{1,n-1}$.  Then $B \in \matrices_n$, 
its entries $b_{ij}$ satisfy $b_{1,n-1} = b_{1n} = b_{2n} = \ldots b_{n-1,n} = 1$ 
(thus $B$ is one-bordered), 
and 
\begin{equation}\label{eq:factors}
\pfaff A = 
\left( \dfrac{a_{1,n-1}}{a_{1n}a_{n-1,n}} \right)^{(n-2)/2} 
\left( \prod_{j=1,\ldots, n-1} a_{jn}  \right) 
\pfaff B .
\end{equation}
\end{propo}
\begin{proof}
Let $A'$ denote the intermediate matrix  
obtained by dividing the $j$-th row and column of $A$ by $a_{jn}$, for all $j=1,\ldots, n-1$.
By \ref{lemma:multiplyrowcol}, $A$' satisfies
\aref{positive} and \aref{quadratic}. 
Its entries $a'_{ij}$ with $i,j\leq n-1$ satisfy 
$a'_{ij} = \dfrac{a_{ij}}{a_{in}a_{jn}}$, hence 
for $j=2,\ldots, n-2$ 
 \[
a'_{1j} 
= \dfrac{a_{1j}}{a_{1n}a_{jn}} \geq  
\dfrac{a_{1,j+1}}{a_{1n}a_{jn}} \geq  
\dfrac{a_{1,j+1}}{a_{1n}a_{j+1,n}} = 
 a'_{1,j+1} ~.
  \]
The last column entries of $A'$ are given by $a'_{jn} = 1$ for $j=1,\ldots, n-1$,
hence $A'$ satisfies \aref{monotone}, and 
$\pfaff A =  \left( \prod_{j=1,\ldots, n-1} a_{jn} \right)\pfaff A'$. 
Now, $B$ is obtained by dividing all the entries of $A'$ in the $(n-1)\times(n-1)$ UL submatrix
by $a'_{1,n-1} = \dfrac{a_{1,n-1}}{a_{1n}a_{n-1,n}}$, and hence 
$\pfaff A' = (a'_{1,n-1})^{(n-2)/2} \pfaff B$, and hence equation \eqref{eq:factors}.
\end{proof}

\begin{propo}
\label{propo:4x4bis}
Let $Q \in \matrices_4$ be 
\[
Q = 
\begin{bmatrix}
0 & Q_{12} & Q_{13} & Q_{14}\\- Q_{12} & 0 & Q_{23} &
Q_{24}\\- Q_{13} & - Q_{23} & 0 & Q_{34}\\
- Q_{14} & - Q_{24} & - Q_{34} & 0\end{bmatrix}.
\]
Then $\pfaff Q \geq \max(Q_{12}Q_{34}, Q_{13}Q_{24}, Q_{14}Q_{23} )$.
\end{propo}
\begin{proof}
As in the proof of \ref{propo:4x4},
by \aref{quadratic},  $Q_{13}Q_{24} \leq  Q_{23}Q_{14}$,
and by \aref{monotone},
$Q_{12}Q_{34} \geq Q_{13}Q_{24}$.
Hence $Q_{13}Q_{24} \leq \min( Q_{23}Q_{14}, Q_{12}Q_{34} )$. 
 and therefore 
$\pfaff Q = Q_{12}Q_{34} - Q_{13}Q_{24} + Q_{14}Q_{23} 
\geq \max(  Q_{23}Q_{14}, Q_{12}Q_{34} ) \geq Q_{13}Q_{24}$, and hence the thesis. 

Another proof is as follows: by \eqref{eq:factors}, 
$\pfaff Q = Q_{24}Q_{13} \pfaff B$, 
where 
\[
B = 
\begin{bmatrix}
0 & b_{12} & 1  & 1 \\- b_{12} & 0 & b_{23} & 1 \\
- 1  & - b_{23} & 0 & 1 \\
- 1  & - 1 & - 1
& 0\end{bmatrix}
\]
and $b_{12} \geq 1 \leq b_{23} $. 
Its Pfaffian is 
$\pfaff B =  b_{12}  - 1 + b_{23} \geq 1$. 
Hence $\pfaff Q \geq Q_{24}Q_{13}$. 
Summarizing, because of  \ref{propo:4x4}, 
\[
\pfaff Q \geq \max(Q_{12}Q_{34}, Q_{13}Q_{24}, Q_{14}Q_{23} ) 
\]
as claimed. 
\end{proof}


Now, the space of all one-bordered skew-symmetric matrices $\bmatrices_n$ 
can be para\-me\-trized by $n(n-3)/2$ variables $v_{i,j}$ as follows. 

\begin{defi}
\label{defi:vij}
Consider the variables 
$v_{x,y}$ with $x,y=1\,\ldots, n-1$, $x<y$; 
assume $v_{1,n-1} = 1$, and $v_{x,y} \geq 1$. 
Define, for simplicity, also 
$v_{j,n} = 1$ for $j=1\ldots n-1$.
The map 
$\mu \from [1,+\infty)^{n(n-3)/2} \to \bmatrices_n$ is defined as follows:
given $v_{ij}$, let $\mu(v_{ij})$ denote the $n\times n$ anti-symmetric 
matrix $Q$ with entries, for $i<j$, 
\[
Q_{ij}  =  \prod_{(x,y) \in  [ 1, \ldots, i] \times [j,\ldots, n ]  }  v_{x,y}.
\]
\end{defi}

Consider that for each even $n$, there are $(1 + 2 + \ldots + n-1) - n = \dfrac{n(n-3)}{2}$ 
non-constant variables  $v_{ij}$. 

For $n=4$, 
\begin{align*}
\begin{bmatrix}
 v_{12} & v_{13}=1  & v_{14}=1\\
  & v_{23} & v_{24}=1\\
 &  & v_{34}=1 
\end{bmatrix}
\mapsto 
\begin{bmatrix}
 0 & v_{12} & 1  &1\\
 -v_{12} & 0 & v_{23} & 1\\
 -1 &  -v_{23} & 0 & 1  \\
 -1 & -1 & -1 & 0 
\end{bmatrix}.
\end{align*}

For $n=6$, 
\begin{align*}
\begin{bmatrix}
  v_{12} & v_{13}  & v_{14} & 1 & 1 \\
         & v_{23}  & v_{24} & v_{25} & 1 \\
         &         & v_{34} & v_{35} & 1 \\
         &         &        & v_{45} & 1 \\
         &         &        &        & 1 
\end{bmatrix}
\mapsto  \mu(v_{ij})
\end{align*}
where
\begin{align}
\label{defi6vij}
\mu(v_{ij}) = 
\begin{bmatrix}
0 & v_{12}v_{13}v_{14} & v_{13}v_{14}  & v_{14} & 1 & 1 \\
 & 0  & v_{13} v_{14} v_{23} v_{24} v_{25}    &  v_{14} v_{24} v_{25}  &  v_{25}   & 1  \\
 &    &  0 &  v_{14} v_{24} v_{25} v_{34} v_{35}   & v_{25}v_{35}    &  1    \\
 &    &    & 0  &  v_{25} v_{35} v_{45}  & 1     \\
 &    &    &    & 0  &  1  \\
 &    &    &    &    & 0   \\
\end{bmatrix}.
\end{align}

\begin{propo}
\label{propo:vij}
Take any choice of $v_{ij}$ in \ref{defi:vij}  ($v_{ij} \geq 1$). Then $\mu(v_{ij}) \in \matrices_n$, 
and it is one-bordered.
Conversely,  
let $B\in \bmatrices_n$ be one-bordered. Then there exist a unique set of variables $v_{x,y}\geq 1$,
for $x=1,\ldots, n$, $y=1,\ldots, n$, $x<y$, 
such that $\mu(v_{ij}) = Q$.   
In other words, 
the map $\mu\from [1,+\infty)^{n(n-3)/2} \to \bmatrices_n$ is a bijection. 
\end{propo}
\begin{proof}
Let $Q=\mu(v_{ij})$. 
By definition,  for $i<j$, 
\[
Q_{ij}  =  \prod_{(x,y) \in  [ 1, \ldots, i] \times [j,\ldots, n ]  }  v_{x,y}. 
\]
Property \aref{positive} is obvious. 
For 
 all the $2\times 2$ submatrices in the upper (positive) triangular half of $Q$  one has
\begin{align*}
\begin{bmatrix} 
Q_{i-1,j} & Q_{i-1,j+1} \\
Q_{i,j} & Q_{i,j+1} 
 \end{bmatrix} 
 =
\begin{bmatrix} 
\prod_{(x,y) \in  [ 1, \ldots, i-1] \times [j,\ldots, n ]  }  v_{x,y} 
& 
\prod_{(x,y) \in  [ 1, \ldots, i-1] \times [j+1,\ldots, n ]  }  v_{x,y} 
\\
\prod_{(x,y) \in  [ 1, \ldots, i] \times [j,\ldots, n ]  }  v_{x,y}
& 
\prod_{(x,y) \in  [ 1, \ldots, i] \times [j+1,\ldots, n ]  }  v_{x,y} 
 \end{bmatrix} 
\end{align*} 
Note that 
if $\alpha = \prod_{x\in [1,\ldots, i-1]} v_{x,j}$ 
  and 
$\beta = \prod_{y\in [j+1,\ldots, n]}  v_{i,y} $, then  
\begin{align*}
Q_{i-1,j} & = Q_{i-1,j+1} \prod_{x\in [1,\ldots, i-1]} v_{x,j}  = Q_{i-1,j+1} \alpha \\
Q_{i,j+1} & = Q_{i-1,j+1} \prod_{y\in [j+1,\ldots, n]}  v_{i,y} = Q_{i-1,j+1} \beta \\
Q_{i,j} & = 
Q_{i-1,j+1} \alpha \beta v_{ij}  
\\
\end{align*}
Hence
\begin{align*}
\det 
\begin{bmatrix} 
Q_{i-1,j} & Q_{i-1,j+1} \\
Q_{i,j} & Q_{i,j+1} 
 \end{bmatrix}  & = 
\det 
\begin{bmatrix} 
 Q_{i-1,j+1} \alpha & Q_{i-1,j+1} \\
 Q_{i-1,j+1} \alpha \beta v_{ij} &  Q_{i-1,j+1} \beta
 \end{bmatrix}  \\
& = Q_{i-1,j+1}^2 (\alpha \beta - \alpha \beta v_{ij}) \leq 0~.
\end{align*}
Thus \aref{quadratic} holds. 
It is easy to see that \aref{monotone} holds, and that it is one-bordered. 

Conversely, 
let $Q\in \matrices_n$ be one-bordered. Proceed as follows: 
First, let $v_{1,n-2} = Q_{1,n-2}$; proceeding for $j=2,\ldots, n-2$, define
$v_{1j} = Q_{ij}/Q_{i,j+1}$, which by \aref{monotone} 
is $v_{1j}\geq 1$ (and unique). 
Similarly, for $i=2\ldots n-2$, let
$v_{i,n-1} = Q_{i,n-1}/Q_{i-1,n-1} \geq 1$.

Now, take a $2\times 2$ submatrix in the upper-right terms, with 
$\alpha$ and $\beta$ as before, i.e. for $2\leq i<j \leq n-1 $.
\begin{align*}
\begin{bmatrix} 
Q_{i-1,j} & Q_{i-1,j+1} \\
Q_{i,j} & Q_{i,j+1} 
 \end{bmatrix}  & = 
\begin{bmatrix} 
 Q_{i-1,j+1} \alpha & Q_{i-1,j+1} \\
 Q_{i-1,j+1} \alpha \beta v_{ij} &  Q_{i-1,j+1} \beta
 \end{bmatrix}  \\
\implies
\alpha & = Q_{i-1,j} / Q_{i-1,j+1} \\
\beta  & = Q_{i,j+1} / Q_{i-1,j+1} \\
v_{ij} & = 
Q_{ij} / ( Q_{i-1,j+1} \alpha \beta ) = 
Q_{ij} / ( Q_{i-1,j+1} \dfrac{Q_{i-1,j}}{Q_{i-1,j+1}}  \dfrac{Q_{i,j+1}}{Q_{i-1,j+1}}  ) = 
\\
 & = Q_{ij} Q_{i-1,j+1} / (  Q_{i-1,j}  Q_{i,j+1})  \geq 1 .
\end{align*}
It is easy to see that 
for such values of $v_{ij}$, $\mu(v_{ij})=Q$, 
and that the choice is unique. 
\end{proof}

With the parametrization of the space of one-bordered matrices $\bmatrices_6$, 
one can prove that their Pfaffians are at least $1$, as follows. Basically,
the Pfaffian is a polynomial in the variables $v_{ij}\geq 1$, and  
by simple algebraic manipulations one can provide a proof.  

\begin{propo}[$n=6$]
\label{propo:n6}
If $B\in \bmatrices_6$ is one-bordered, then $\pfaff B \geq 1$. 
\end{propo}
\begin{proof}
Consider the substitution of variables $v_{ij}\geq 1$ defined above, as in \eqref{defi6vij}.
Expanding on the last column (and Knuth's notation: the hat means deleted column-row)
\begin{align*}
\pfaff B &= B[\hat 1,2,3,4,5,\hat 6] 
-  B[1,\hat 2,3, 4, 5,\hat 6 ] \\
& +  B[1,2,\hat 3,4,5,\hat 6 ] \\
& -  B[1,2,3,\hat 4, 5, \hat 6]
+  B[1,2,3,4,\hat 5,\hat 6]
\\
& = 
(B_{23}B_{45} - B_{24}B_{35} + B_{25}B_{34})
- ( B_{13} B_{45} - B_{14}B_{35} + B_{34} ) \\
& + ( B_{12} B_{45} - B_{14} B_{25} + B_{24} )  \\
& -  ( B_{12} B_{35} - B_{13}B_{25} + B_{23}  )
 +  (B_{12}B_{34} - B_{13}B_{24} + B_{14}B_{23} )~.
\\
\end{align*}
Now, $ B[\hat 1,2,3,4,5,\hat 6] 
-  B[1,\hat 2,3, 4, 5,\hat 6 ] $ can be written as  
\begin{align*}
& (B_{23}B_{45} - B_{24}B_{35} + B_{25}B_{34})
- ( B_{13} B_{45} - B_{14}B_{35} + B_{34} ) \\
& = 
(B_{23}B_{45} - B_{24}B_{35})
- ( B_{13} B_{45} - B_{14}B_{35} )
+ ( B_{25}
- 1 ) B_{34}   \\
& \geq 
B_{23}B_{45} - B_{24}B_{35}
-  B_{13} B_{45} + B_{14}B_{35} 
\end{align*}
By using $v_{ij}\geq 1 $ variables, 
$ B_{23}B_{45} - B_{24}B_{35}
-  B_{13} B_{45} + B_{14}B_{35}  $ can be written as 
\begin{align*}
& v_{13} v_{14} v_{23} v_{24} v_{25}  v_{25} v_{35} v_{45} - v_{14} v_{24} v_{25} v_{25}v_{35}   
 -  v_{13}v_{14}  v_{25} v_{35} v_{45} + v_{14}  v_{25}v_{35}\\
 & = 
v_{14} v_{25} v_{35} (
v_{13}  v_{23} v_{24}   v_{25}  v_{45} -  v_{24}  v_{25} 
-  v_{13}   v_{45} + 1 ) \\
& \geq  
v_{13}  v_{24}   v_{25}  v_{45} -  v_{24}  v_{25} 
-  v_{13}   v_{45} + 1  \\
& = (v_{13} v_{45} - 1) ( v_{24} v_{25} - 1 ) \geq 0 
\end{align*} 
and hence 
 $ B[\hat 1,2,3,4,5,\hat 6] 
-  B[1,\hat 2,3, 4, 5,\hat 6 ] \geq 0 $.

Furthermore, 
$B[1,2,3,4,\hat 5,\hat 6]
 -  B[1,2,3,\hat 4, 5, \hat 6]$
 can be written as
\begin{align*}
&    (B_{12}B_{34} - B_{13}B_{24} + B_{14}B_{23} )
 -  ( B_{12} B_{35} - B_{13}B_{25} + B_{23}  ) \\ 
& = 
   (B_{12}B_{34} - B_{13}B_{24} )
 -  ( B_{12} B_{35} - B_{13}B_{25} ) 
   + ( B_{14}-1) B_{23} 
   \\
 &   \geq 
   B_{12}B_{34} - B_{13}B_{24} 
 -   B_{12} B_{35} + B_{13}B_{25} 
\end{align*}
By using $v_{ij}\geq 1 $ variables, 
$ B_{12}B_{34} - B_{13}B_{24} 
 -   B_{12} B_{35} + B_{13}B_{25}  $ 
can be written as
\begin{align*}
& v_{12}v_{13}v_{14} v_{14} v_{24} v_{25} v_{34} v_{35} - v_{13}v_{14} v_{14} v_{24} v_{25} 
- v_{12}v_{13}v_{14} v_{25}v_{35} +  v_{13} v_{14} v_{23} v_{24} v_{25}  v_{25}  
\\
& = v_{13}v_{14} v_{25} ( 
v_{12} v_{14} v_{24}  v_{34} v_{35} -  v_{14} v_{24}  
- v_{12} v_{35} +   v_{23} v_{24} v_{25}    
) \\
& \geq 
v_{12} v_{35} v_{14} v_{24} - v_{14} v_{24} 
- v_{12} v_{35} + 1 \\
& = (v_{12}v_{35} - 1 ) (v_{14} v_{24} - 1) \geq 0 ,
\end{align*}
and hence 
$ B_{12}B_{34} - B_{13}B_{24} 
 -   B_{12} B_{35} + B_{13}B_{25}  \geq 0$.

Finally,
$ B_{12} B_{45} - B_{14} B_{25} + B_{24}  $
can be written as
\begin{align*}
& 
v_{12}v_{13}v_{14}  v_{25} v_{35} v_{45} - v_{14} v_{25} +  v_{14} v_{24} v_{25} \\
& = v_{14} v_{25} ( 
v_{12}v_{13}v_{35} v_{45} - 1 +  v_{14} v_{24} ) \geq 1,  
\end{align*} 
which means 
$ B_{12} B_{45} - B_{14} B_{25} + B_{24} \geq 1 $.
We can conclude therefore that $\pfaff B \geq 1$, as claimed.  
%
\end{proof}

Instead of algebraic manipulations and inequalities $v_{ij}\geq 1$,
another idea is to define the non-negative variables $u_{ij} = v_{ij}-1$, 
parametrize $\bmatrices_n$ in terms of $u_{ij}$, and write the 
generic Pfaffian as a polynomial in the variables $u_{ij}$.  
What (remarkably) happens, is that such Pfaffians appear to be polynomials
in $u_{ij}$ with 
only non-negative coefficients. 

\begin{defi}
\label{defi:uij}
For each $v_{ij}$ (with $1\leq i < j \leq n $)  
in \ref{defi:vij}, let $u_{ij} = v_{ij} - 1$; clearly 
it follows that $v_{ij} \geq 1 \iff u_{ij} \geq 0$. 
\end{defi}

\begin{propo}
\label{propo:uij}
Consider the map $p_n \from [0,+\infty)^{n(n-3)/2} \to \RR$ 
defined as the composition in the following diagram:
\[
 \xymatrix@C+12pt@R+12pt{%
 { [0,+\infty)^{n(n-3)/2} }  
\ar@/_2em/[rrr]_-{p_n}
 \ar[r]^-{u\mapsto v}_-{\approx} 
 &  { [1,\infty)^{n(n-3)/2} } \ar[r]^-{\mu}_-{\approx} &  
 {\bmatrices_b} \ar[r]^-{\pfaff} &  { \RR  } 
 }
\]
In other words, if $Q\in \bmatrices_n$, and $v_{i,j}$ are the variables 
(defined in \ref{defi:vij}) of proposition \ref{propo:vij}, then by substituting 
$v_{ij} \leftarrow 1+u_{ij}$ the coefficients $Q_{ij}$ become polynomials,
with integer coefficients, in the variables
$u_{ij}$. 
The composition $p_n(u)  = \pfaff Q $ is  a polynomial (with integer coefficients)
in the $\frac{n(n-3)}{2}$ variables $u_{ij}$.
\end{propo}
\begin{proof}
The Pfaffian of $Q$ is a polynomial in the variables $Q_{ij}$. Both 
$u \mapsto v = 1+u$ and $\mu$ are polynomial functions, 
therefore the claim. 
\end{proof}

\begin{defi}
\label{defi:succeq}
Consider the following order relation for polynomials $f(x)$ in variables $x_1, \ldots, x_l$:
$f \succeq g \iff f-g$ has non-negative coefficients, i.e. every coefficient of the polynomial $f-g$ is non-negative.
\end{defi}

Note that $f \succeq 0 $ if and only if $f$ has only non-negative coefficients,
and therefore 
 $\forall x_1\geq 0, \ldots, x_l\geq 0$, 
 $f(x_1,\ldots, x_l) \geq f(0) \geq 0$.

\begin{lemma}
\label{lemma:fgeq1}
If $f(x)\succeq 1$, then for all positive values of the variables $x_j$ 
the inequality $f(x) \geq 1$ holds. 
\end{lemma}
\begin{proof}
The polynomial $f(x) - 1$ has only positive non-zero coefficients, hence 
$f(0) \geq 1$. 
Moreover, for non-negative values of the $x_j$ the inequality $f(x) - 1 \geq f(0)-1\geq 0$ holds.
\end{proof}

From now on, the relation $\succeq$ is meant to be in $\RR[u_{ij}]$, 
so that when polynomials in other variables (such as $Q_{ij}$ or $v_{ij}$)
are handled, the relation holds once all the proper substitutions are 
in place, and the resulting polynomials have variables $u_{ij}$. 
This is important, because sometimes it is simple to write 
a polynomial in $v_{ij}$ which in reality represents a polynomial 
in $u_{ij}$.

\begin{propo}
\label{propo:uij6}
For $n=4,6$, the Pfaffian polynomial $p_n(u)$ of proposition \ref{propo:uij}
satisfies the inequality 
$ p_n(u) \succeq 1$.
\end{propo}
\begin{proof}
For $n=4$, 
\begin{align*}
Q= \begin{bmatrix}
 0 & v_{12} & 1  &1\\
 -v_{12} & 0 & v_{23} & 1\\
 -1 &  -v_{23} & 0 & 1  \\
 -1 & -1 & -1 & 0 
\end{bmatrix},
\end{align*}
$\implies$
$\pfaff Q = v_{12} - 1 + v_{23} = (u_{12} +1) -1  + (u_{23}+1) \succeq 1$. 

In the same way, for $n=6$, 
consider the same steps in the proof of \ref{propo:n6}. 
Each $B_{ij}$ satisfies $B_{ij} \succeq 1$,  
hence $(B_{25} - 1)B_{34} \succeq 0$. Furthermore,
each $v_{ij}$ (as polynomial in $u_{ij}$) satisfies $v_{ij} \succeq 1$, 
therefore one can follow all the same steps in the proof of \ref{propo:n6} 
to arrive at the conclusion that $p_6(u) \succeq 1$.  
\end{proof}

\begin{remark}
\label{remark:otherUL}
Consider the 1-bordered $6\times 6$ matrix 
\[\small
\begin{bmatrix}
0 & Q_{12} & Q_{13} & Q_{14} & 1 & 1\\
- Q_{12} & 0 & Q_{23} & Q_{24} & Q_{25} & 1\\
- Q_{13} & - Q_{23} & 0 & Q_{34} & Q_{35} & 1\\
- Q_{14} & - Q_{24} & - Q_{34} & 0 & Q_{45} & 1\\
-1 & - Q_{25} & - Q_{35} & - Q_{45} & 0 & 1\\
-1 & -1 & -1 & -1 & -1 & 0
\end{bmatrix}
\]
With elementary row and column operations can be written (use the entries in $(5,6)$ and $(6,5)$ 
to set all the resulting zeros) as follows:
\[\small
\begin{bmatrix}0 & Q_{12} + Q_{25} - 1 & Q_{13} + Q_{35} - 1 & Q_{14} + Q_{45} - 1 & 0 & 0\\
- Q_{12} - Q_{25} + 1 & 0 & Q_{23} - Q_{25} + Q_{35} & Q_{24} - Q_{25} + Q_{45} & 0 & 0\\
- Q_{13} - Q_{35} + 1 & - Q_{23} + Q_{25} - Q_{35} & 0 & Q_{34} - Q_{35} + Q_{45} & 0 & 0\\
- Q_{14} - Q_{45} + 1 & - Q_{24} + Q_{25} - Q_{45} & - Q_{34} + Q_{35} - Q_{45} & 0 & 0 & 0\\
0 & 0 & 0 & 0 & 0 & 1\\
0 & 0 & 0 & 0 & -1 & 0
\end{bmatrix}
\]
and the question is whether the $4\times 4$ UL submatrix
belongs to $\matrices_4$, when $Q$ belongs to $\matrices_6$. Unfortunately, this is not the case,
in general. 
For example, 
the following matrix belongs to $\matrices_6$
\[
\begin{bmatrix}0 & 4 & 3 & 2 & 1 & 1\\
-4 & 0 & 8 & 5 & 2 & 1\\
-3 & -8 & 0 & 11 & 4 & 1\\
-2 & -5 & -11 & 0 & 5 & 1\\
-1 & -2 & -4 & -5 & 0 & 1\\
-1 & -1 & -1 & -1 & -1 & 0
\end{bmatrix}
\]
while the UL submatrix after the above row and columns operations is
\[
\begin{bmatrix}0 & 5 & 6 & 6\\
-5 & 0 & 10 & 8\\
-6 & -10 & 0 & 12\\
-6 & -8 & -12 & 0
\end{bmatrix}
\]
which does not satisfy \aref{monotone}. As show in remark \ref{remark:counterexamplen6}, 
property \aref{monotone} cannot be dropped from definition \ref{defi:T}. 
\end{remark}

\section{Symbolic Computations on Pfaffian Polynomials in Variables $u_{ij}$ and $v_{ij}$}
\label{section:3} 

The Pfaffian of a $n\times n$ one-bordered matrix $B$ is the alternating
sum of $(n-1)!!$ monomials in $B_{ij}$, one for each partition of $\{1,\ldots, n\}$. 
After the substitution of \ref{defi:vij}, 
each monomial is transformed into a monomial in the variables $v_{ij}$, 
and the Pfaffian is written as a polynomial with integer coefficients, in the variables
$v_{ij}$, with $(n-1)!!$ terms.  
The polynomial $p_n(u)$ of \ref{propo:uij}
is obtained by substituting $v_{ij} = 1 + u_{ij}$. 
It was shown in Proposition \ref{propo:uij6} that $p_n(u)\succeq 1$ for $n=4$ 
and $n=6$. We are going to show that $p_n(u) \succeq 1$ also for some values 
of $n\geq 8$. 
The problem is that the number of terms in $p_n(u)$ grows very rapidly.

For $n=8$, the polynomial $p_8(u)$ of \ref{propo:uij} 
in $8(8-3)/2=20$ variables has 6727104  coefficients, and of course 
they cannot be computed by hand. Even for a machine,
computing \emph{all} the coefficients of $p_8(u)$ is not a fast task. 
Defining a polynomial for $n\geq 10$ is not reasonable, as a general approach.  
Trying to do so, results in a out-of-memory exception for the running python. 
For $n\geq 10$,  
a direct symbolic computation of $p_n(u)$  (in the ring of polynomials with integer coefficients) 
fails, on all major CAS I have tried (maple, mathematica, MAGMA, python-sympy).

Hence, we introduce an algorithm for proving 
that $p_n(u) \succeq 1$ \emph{without} actually knowing all its
coefficients at once. 
The idea is to operate on the original polynomial in the variables $v_{ij}$, 
and, without expanding the substitution $v_{ij} \leftarrow 1+u_{ij}$, to simplify 
and eliminate the variables $v_{ij}$ one at the time, until it is trivial 
to prove that the final polynomial $f(v)$ satisfies $f(1+u) \succeq 1$. 
The final trivial case is as follows: 
given a generic polynomial $Q(\vv)$ with integer coefficients in the variables $v_{ij}$,
write it as $q_+(\vv) - q_-(\vv)$, 
where both $q_+(\vv)$ and $q_-(\vv)$ have non-negative coefficients. 
If one of $q_+$ or $q_-$ is zero, or if $q_-$ has degree $0$, 
it is easy to determine if $q_+(1+\vu) - q_-(1+\vu) \succeq 1$.  

The procedure used to eliminate the variables $v_{ij}$, one at the time, 
is as follows. 
If there is at least a variable $v_{ij}$ of degree $1$, then consider the following lemma (recall that the relation $\succeq$ is evaluated after all the substitutions
$v_{ij} \leftarrow 1+u_{ij}$ have been performed). 
\begin{lemma}
\label{lemma:L1}
Let $p(v,\vu)$ be a polynomial in variables $v,u_1,\ldots, v_l$,
of degree $1$ in the variable $v$.  
If $p(v,\vu) = a(\vu) v + b(\vu)$, where $a(\vu) \succeq 0$, 
then  $p(1+u,\vu) \succeq a(\vu) + b(\vu)$, and therefore
$p(1+u,\vu) \succeq L $ $\iff$ $a(\vu) + b(\vu) \succeq L$.  
\end{lemma}
\begin{proof} 
Since $p(1+u,\vu)  = a(\vu) (1+u) + b(\vu) = a(\vu) u + (a(\vu) + b(\vu)$, 
if $a(\vu) \succeq 0$, then 
  $a(\vu) u \succeq 0 \implies p(1+u,\vu) - a(\vu) - b(\vu) \succeq 0$,  
 and hence 
 $p(v,\vu) \succeq a(\vu) + b(\vu)$. 
This implies that if $a(\vu) + b(\vu) \succeq L$, 
then $p(v,\vu) \succeq L$. Conversely, 
if $p(v,\vu) \succeq L$, it must be $a(\vu) \succeq 0$ (which is the coefficient of $u$) and 
  $a(\vu) + b(\vu) \succeq L$ (the term of degree 0 in $u$). 
\end{proof}

By applying Lemma \ref{lemma:L1} a iteratively, one can set $v_{ij} = 1$
for all variables of degree $1$, 
provided the relative coefficients $a$ satisfy $a\succeq 0$. If $a\not\succeq 0$, 
then the claim is false, and the procedure ends in a failure.  
If there are no variables of degree $1$, 
but there is at least a variable $v_{ij}$ of degree $2$, then consider the following lemma. 

\begin{lemma}
\label{lemma:L2}
Let $p(v,\vu)$ be a polynomial in variables $v,u_1,\ldots, v_l$,
of degree $2$ in the variable $v$.  
If $p(v,\vu) = a(\vu) v^2 + b(\vu) v + c(\vu$, where $a(\vu) \succeq 0$, 
$2a(\vu) + b(\vu) \succeq 0$, 
then  $p(1+u,\vu) \succeq a(\vu) + b(\vu) + c(\vu)$, and therefore
$p(1+u,\vu) \succeq L $ $\iff$ $p(1,\vu) = a(\vu) + b(\vu) \succeq L$. 

More generally, if $p(v,\vu)$ has degree $d$ in $v$, 
\[
p(v, \vu)  = \sum_{k=0}^d a_k(\vu) v^k
\]
if  for $j=1,\ldots, d$ 
\[
\sum_{k = j \ldots d } a_k(\vu) \binom{k}{j} \succeq 0 
\]  
then $p(1+u,\vu) \succeq L$ 
$\iff$ 
$ p(1,\vu) = \sum_{k=0}^d a_k(\vu) \succeq L $. 
\end{lemma}
\begin{proof} 
Since $p(1+u,\vu) = a(\vu) (1+u)^2 + b(\vu) (1+u) + c(\vu) 
= a(\vu) u^2 + (2a(\vu) + b(\vu) ) u + (a(\vu) + b(\vu) + c(\vu))$, 
if $a(\vu) \succeq 0$ and $2a(\vu) + b(\vu) \succeq 0$, 
 then  $a(\vu) u^2 \succeq 0 $, 
 $(2a(\vu) + b(\vu))u \succeq 0$, 
 and therefore 
$p(1+u,\vu) \succeq a(\vu) + b(\vu) + c(\vu)$.
If $a(\vu)+b(\vu)+c(\vu) \succeq L$, then 
$p(v,\vu) \succeq L$. 
Conversely,
if 
$p(1+u,\vu) \succeq L$,
then 
$a(\vu) \succeq 0$  (the coefficient of $u^2$),
$2a(\vu) + b(\vu) \succeq 0$ (the coefficient of $u$), 
and  the terms of degree $0$ in $u$ has to satisfy
$a(\vu) + b(\vu) + c(\vu) \succeq L$. 

More generally, 
\begin{align*}
p(1+u, \vu)  & = \sum_{k=0}^d a_k(\vu) (1+u)^k \\ 
& = \sum_{k=0}^d a_k(\vu) \left( \sum_{j=0}^k \binom{k}{j} u^j  \right) \\
& = \sum_{j=0}^d \left(
\sum_{k=j\ldots d} a_k(\vu) \binom{k}{j} 
\right) u^j, 
\end{align*}
hence $p(1+u,\vu) \succeq L$ 
$\iff$ 
for $j=1\ldots d$ the inequality 
$\sum_{k=j\ldots d} a_k(\vu) \binom{k}{j} \succeq 0$ holds,
and (for $j=0$) the inequality
$\sum_{k=0\ldots d} a_k(\vu) \binom{k}{0} \succeq L$ holds,  
as claimed. 
\end{proof}

By applying Lemma \ref{lemma:L2} a iteratively, one can set $v_{ij} = 1$ 
(and thus $u_{ij}=0$, in the corresponding $p_n(u)$)
for all variables of degree $2$, 
provided the relative coefficients $a,b$ satisfy $a\succeq 0$
and $2a+b\succeq 0$. If $a\not\succeq 0$, or $2a+b\not\succeq 0$, 
then the claim is false, and the procedure ends in a failure.  
If there are no variables of degree $1$ or $2$, 
the same procedure can be applied to terms of degree $3$, and so on. 
To avoid overloading the branching of too many function calls, 
the following lemma can be applied for degree $4$ and bigger. 
It is necessary to iterate over all the monomials dividing 
all the terms $q_-(\vv)$ (the terms with negative coefficients in the
polynomial $Q$), which computationally does not need to allocate all the coefficients 
of the polynomials in memory. 

\begin{lemma}
\label{lemma:L3}
Let $Q(\vv) = \sum_{k=1}^l \alpha_k(\vv) - \sum_{k=1}^{l'} \beta_k(\vv)$ be a polynomial
in variables $\vv=(v_1,v_2, \ldots)$, where for each index $k$ the terms $\alpha_k$ and $\beta_k$ are monomials 
in $\vv$. Let $d_{k,i}$ and $d'_{k,i}$ the indices such that 
$\alpha_k(\vv) = \prod_i v_i^{d_{k,i}}$
and 
$\beta_k(\vv) = \prod_i v_i^{d'_{k,i}}$.
Then 
$Q(1+\vu) \succeq 0$ $\iff$
for each monomial $\prod_{i} v_{i}^{j_i}$ dividing 
  at least one of the  terms $\beta_k(\vv)$, the following inequality holds:
\[
\sum_{k=1}^l \prod_{i} \binom{d_{k,i}}{j_i} \geq  \sum_{k=1}^{l'} \prod_{i} \binom{d'_{k,i}}{j_i}~. 
\]
\end{lemma}
\begin{proof} 
In the following equations, $j_* = \zero \ldots d_{k,*}$ means 
that the vector $j_* = (j_1, j_2, \ldots)$ iterates over the Cartesian product
$[0,d_{k,1}] \times [0,d_{k,2}] \ldots $, which means for each $i$ the iteration 
runs over $j_i = 0 \ldots d_{k,i}$.  
\begin{align*}
Q(1+\vu)  &  = 
\sum_{k=1}^l \alpha_k(1+\vu) - \sum_{k=1}^{l'} \beta_k(1+\vu) \\
& = 
\sum_{k=1}^l  \prod_i (1+u_i)^{d_{k,i}} - 
\sum_{k=1}^{l'}  \prod_i (1+u_i)^{d'_{k,i}} \\
& = 
\sum_{k=1}^l  \prod_i \left( \sum_{j_i=0}^{d_{k,i}}   \binom{d_{k,i}}{j_i} u_i^{j_i} \right) -  
\sum_{k=1}^{l'}  \prod_i \left( \sum_{j_i=0}^{d'_{k,i}}   \binom{d'_{k,i}}{j_i} u_i^{j_i} \right)  \\ 
& = 
\sum_{k=1}^l \sum_{j_* = \zero \ldots d_{k,*} }  \prod_{i}\left( \binom{d_{k,i}}{j_i}  u_i^{j_i} \right) 
-
\sum_{k=1}^{l'} \sum_{j_* = \zero \ldots d'_{k,*} }  \prod_{i}\left( \binom{d'_{k,i}}{j_i}  u_i^{j_i} \right) .
\end{align*}
Thus, the monomial terms in $Q(1+\vu)$ with possible negative coefficients 
are those  written as $\prod_{i}\binom{d'_{k,i}}{j_i}  u_i^{j_i} $
with $0 \leq j_i \leq d'_{k,i}$ for each $i$ and some $k$, which means they are the exponents 
of a monomial $\prod_{i} v_{i}^{j_i}$ dividing 
  at least one of the  terms $\beta_k(\vv)$. Its coefficient is 
\[
\sum_{k=1}^l\prod_{i} \binom{d_{k,i}}{j_i}    - 
\sum_{k=1}^{l'} \prod_{i} \binom{d'_{k,i}}{j_i} , 
\]
which concludes the proof. 
\end{proof}

The algorithm goes as follows. Let $P(v_{ij})$ be
a polynomial in the variables 
$v_{ij}$. 
One at the time, eliminate all the terms
$v_{ij}$ of degree $1$, $2$ and $3$  by Lemma \ref{lemma:L1} and \ref{lemma:L2}. 
If any elimination fails, it means that $P(1+u_{ij})\not\succeq 0$. 
In case all the terms left have degree at least $4$,
apply \ref{lemma:L3}. 
Thus the algorithm can terminate with 
a positive result (if $Q(\vv) \succeq 1$) or a failure. 
A worked-out case ($n=6$) is written below explicitly, in 
remark \ref{rem:n6}. 

\begin{theo}[Computer-assisted]
\label{theo:n14}
For even $n\leq 14$, the Pfaffian polynomial $p_n(u)$ defined in \ref{propo:uij} 
satisfies $p_n(u) \succeq 1$. 
\end{theo}

As a consequence, the following Corollary holds. 
\begin{coro}[Computer-assisted] 
\label{coro:mainineq}
For even $n \leq 14$, 
let $Q \in \matrices_n$.  
Then the following inequality holds,
\begin{equation}
\label{eq:mainineq}
\pfaff Q \geq 
\left( \dfrac{Q_{1,n-1}}{Q_{1n}Q_{n-1,n}} \right)^{(n-2)/2} 
\left( \prod_{j=1,\ldots, n-1} Q_{jn}  \right),
\end{equation}
and therefore 
the Pfaffian of $Q$ is strictly positive. 
\end{coro}
\begin{proof}
By \ref{propo:factors}, there exists a one-bordered matrix $B \in \bmatrices_n$
such that 
\[ 
\pfaff Q  =
\left( \dfrac{Q_{1,n-1}}{Q_{1n}Q_{n-1,n}} \right)^{(n-2)/2} 
\left( \prod_{j=1,\ldots, n-1} Q_{jn}  \right)   \pfaff B. 
\]
By \ref{propo:uij} there exist
values of the variables $u_{ij}$ such that 
$\pfaff B = p_n(u)$, where  
$p_n$ is the Pfaffian polynomial of proposition \ref{propo:uij}.
Now, by theorem \ref{theo:n14}, $p_n \succeq 1$ for all even $n\leq 14$, 
and therefore by \ref{lemma:fgeq1} the inequality $p_n(u) \geq 1$ holds.  
\end{proof}

\begin{remark}
The running times (python interpreter \verb+pypy3+ on a PC with Intel Xeon CPU
E3-1245 v5 @ 3.50GHz)  for $n=8\ldots 14$ are listed in the following table: 

\begin{tabular}{|l|l|}
\hline 
n & \text{running time HH:MM:SS} \\
\hline
8 & 00:00:01 \\
10 & 00:00:02 \\ 
12 & 00:01:56 \\
14 & 103:49:30\\
\hline
\end{tabular}
\end{remark}

\begin{remark}
\label{rem:n6}
For $n=6$ the algorithm terminates by elimination of terms just of degree $1$,
without the application
of Lemma \ref{lemma:L2} or Lemma \ref{lemma:L3}.  We follow the calculations step-by-step as follows. 
\begin{dmath*}\small
Q = \left[\begin{matrix}0 & v_{12} v_{13} v_{14} & v_{13} v_{14} & v_{14} & 1 & 1\\- v_{12} v_{13} v_{14} & 0 & v_{13} v_{14} v_{23} v_{24} v_{25} & v_{14} v_{24} v_{25} & v_{25} & 1\\- v_{13} v_{14} & - v_{13} v_{14} v_{23} v_{24} v_{25} & 0 & v_{14} v_{24} v_{25} v_{34} v_{35} & v_{25} v_{35} & 1\\- v_{14} & - v_{14} v_{24} v_{25} & - v_{14} v_{24} v_{25} v_{34} v_{35} & 0 & v_{25} v_{35} v_{45} & 1\\-1 & - v_{25} & - v_{25} v_{35} & - v_{25} v_{35} v_{45} & 0 & 1\\-1 & -1 & -1 & -1 & -1 & 0\end{matrix}\right]
\end{dmath*}
\begin{dmath*}
\pfaff Q = v_{12} v_{13} v_{14}^{2} v_{24} v_{25} v_{34} v_{35} + v_{12} v_{13} v_{14} v_{25} v_{35} v_{45} - v_{12} v_{13} v_{14} v_{25} v_{35} + v_{13} v_{14}^{2} v_{23} v_{24} v_{25} - v_{13} v_{14}^{2} v_{24} v_{25} + v_{13} v_{14} v_{23} v_{24} v_{25}^{2} v_{35} v_{45} - v_{13} v_{14} v_{23} v_{24} v_{25} - v_{13} v_{14} v_{25} v_{35} v_{45} + v_{13} v_{14} v_{25} + v_{14} v_{24} v_{25}^{2} v_{34} v_{35} - v_{14} v_{24} v_{25}^{2} v_{35} - v_{14} v_{24} v_{25} v_{34} v_{35} + v_{14} v_{24} v_{25} + v_{14} v_{25} v_{35} - v_{14} v_{25} 
\end{dmath*}
After collecting $v_{14} v_{25}$ the Pfaffian can be written as
 \begin{dmath*}
   \pfaff Q = v_{14} v_{25} \left( v_{12} v_{13} v_{14} v_{24} v_{34} v_{35} + v_{12} v_{13} v_{35} v_{45} + v_{13} v_{14} v_{23} v_{24} + v_{13} v_{23} v_{24} v_{25} v_{35} v_{45} + v_{13} + v_{24} v_{25} v_{34} v_{35} + v_{24} + v_{35} - (v_{12} v_{13} v_{35} + v_{13} v_{14} v_{24} + v_{13} v_{23} v_{24} + v_{13} v_{35} v_{45} + v_{24} v_{25} v_{35} + v_{24} v_{34} v_{35} + 1) \right)  
   \end{dmath*}
It is clear that the term $v_{14} v_{25}$ can be dropped; 
this simple step is performed every time there is such a factor, and it
makes the procedure a little bit faster. 
 Let $P,p$ denote the following polynomials in the 9 $v$-variables 
 $v_{12}$, $v_{13}$, $v_{14}$, $v_{23}$, $v_{24}$, $v_{25}$, $v_{34}$, $v_{35}$, $v_{45}$: 
\begin{dgroup*}
\begin{dmath*}
P  = v_{12} v_{13} v_{14} v_{24} v_{34} v_{35} + v_{12} v_{13} v_{35} v_{45} + v_{13} v_{14} v_{23} v_{24} + v_{13} v_{23} v_{24} v_{25} v_{35} v_{45} + v_{13} + v_{24} v_{25} v_{34} v_{35} + v_{24} + v_{35}  
\end{dmath*}
\begin{dmath*}
p  = v_{12} v_{13} v_{35} + v_{13} v_{14} v_{24} + v_{13} v_{23} v_{24} + v_{13} v_{35} v_{45} + v_{24} v_{25} v_{35} + v_{24} v_{34} v_{35} + 1 
\end{dmath*}
\end{dgroup*}
Now, $P=Av_{14}+B$ and $p=av_{14} +b$ where 
\begin{dgroup*}
\begin{dmath*}
A  = v_{13} v_{24} (v_{12} v_{34} v_{35} + v_{23})
\end{dmath*}
\begin{dmath*}
a  = v_{13} v_{24},
\end{dmath*}
\end{dgroup*}
(and both $B$ and $b$ do not have the variable $v_{14}$).
Since $A-a \succeq 1$, by \ref{lemma:L1}
\begin{dmath*}
\pfaff Q  \succeq  A+B -a-b = v_{12} v_{13} v_{24} v_{34} v_{35} + v_{12} v_{13} v_{35} v_{45} + v_{13} v_{23} v_{24} v_{25} v_{35} v_{45} + v_{13} + v_{24} v_{25} v_{34} v_{35} + v_{24} + v_{35} - (v_{12} v_{13} v_{35} + v_{13} v_{24} + v_{13} v_{35} v_{45} + v_{24} v_{25} v_{35} + v_{24} v_{34} v_{35} + 1).
\end{dmath*}
The variable $v_{14}$ has been eliminated. 
Now, let $P,p$ denote the following polynomials in the remaining 8 $v$-variables 
 $v_{12}$, $v_{13}$, $v_{23}$, $v_{24}$, $v_{25}$, $v_{34}$, $v_{35}$, $v_{45}$: 
\begin{dgroup*}
\begin{dmath*}
P  = v_{12} v_{13} v_{24} v_{34} v_{35} + v_{12} v_{13} v_{35} v_{45} + v_{13} v_{23} v_{24} v_{25} v_{35} v_{45} + v_{13} + v_{24} v_{25} v_{34} v_{35} + v_{24} + v_{35}  
\end{dmath*}
\begin{dmath*}
p  = v_{12} v_{13} v_{35} + v_{13} v_{24} + v_{13} v_{35} v_{45} + v_{24} v_{25} v_{35} + v_{24} v_{34} v_{35} + 1 
\end{dmath*}
\end{dgroup*}
As above, $P=Av_{23}+B$ and $p=b$ where 
\begin{dgroup*}
\begin{dmath*}
A  = v_{13} v_{24} v_{25} v_{35} v_{45} 
\end{dmath*}
\end{dgroup*}
and both $B$ and $b$ do not have the variable $v_{23}$.
Since $A \succeq 1$, by \ref{lemma:L1}
\begin{dmath*}
\pfaff Q  \succeq  A+B -b = v_{12} v_{13} v_{24} v_{34} v_{35} + v_{12} v_{13} v_{35} v_{45} + v_{13} v_{24} v_{25} v_{35} v_{45} + v_{13} + v_{24} v_{25} v_{34} v_{35} + v_{24} + v_{35} - (v_{12} v_{13} v_{35} + v_{13} v_{24} + v_{13} v_{35} v_{45} + v_{24} v_{25} v_{35} + v_{24} v_{34} v_{35} + 1)
\end{dmath*}
Let $P,p$ denote the following polynomials in the remaining 7 $v$-variables: 
\begin{dgroup*}
\begin{dmath*}
P  = v_{12} v_{13} v_{24} v_{34} v_{35} + v_{12} v_{13} v_{35} v_{45} + v_{13} v_{24} v_{25} v_{35} v_{45} + v_{13} + v_{24} v_{25} v_{34} v_{35} + v_{24} + v_{35}  
\end{dmath*}
\begin{dmath*}
p  = v_{12} v_{13} v_{35} + v_{13} v_{24} + v_{13} v_{35} v_{45} + v_{24} v_{25} v_{35} + v_{24} v_{34} v_{35} + 1 
\end{dmath*}
\end{dgroup*}
Proceeding as above, 
$P=Av_{12}+B$ and $p=av_{12} +b$ where 
\begin{dgroup*}
\begin{dmath*}
A  = v_{13} v_{35} (v_{24} v_{34} + v_{45})
\end{dmath*}
\begin{dmath*}
a  = v_{13} v_{35} 
\end{dmath*}
\end{dgroup*}
and both $B$ and $b$ do not have the variable $v_{12}$.
Since $A-a \succeq 1$, by \ref{lemma:L1}
\begin{dmath*}
\pfaff Q  \succeq  A+B -a-b = v_{13} v_{24} v_{25} v_{35} v_{45} + v_{13} v_{24} v_{34} v_{35} + v_{13} + v_{24} v_{25} v_{34} v_{35} + v_{24} + v_{35} - (v_{13} v_{24} + v_{13} v_{35} + v_{24} v_{25} v_{35} + v_{24} v_{34} v_{35} + 1)
\end{dmath*}
Furthermore, let $P,p$ denote the following polynomials in the remaining 6 $v$-variables: 
\begin{dgroup*}
\begin{dmath*}
P  = v_{13} v_{24} v_{25} v_{35} v_{45} + v_{13} v_{24} v_{34} v_{35} + v_{13} + v_{24} v_{25} v_{34} v_{35} + v_{24} + v_{35}  
\end{dmath*}
\begin{dmath*}
p  = v_{13} v_{24} + v_{13} v_{35} + v_{24} v_{25} v_{35} + v_{24} v_{34} v_{35} + 1 
\end{dmath*}
\end{dgroup*}
Now, we eliminate $v_{45}$ by setting $P=Av_{45}+B$ and $p=b$ where 
\begin{dgroup*}
\begin{dmath*}
A  = v_{13} v_{24} v_{25} v_{35}
\end{dmath*}
\end{dgroup*}
and both $B$ and $b$ do not have the variable $v_{45}$.
Since $A\succeq 1$, by \ref{lemma:L1}
\begin{dmath*}
\pfaff Q  \succeq  A+B -b 
= v_{13} v_{24} v_{25} v_{35} + v_{13} v_{24} v_{34} v_{35} + v_{13} + v_{24} v_{25} v_{34} v_{35} + v_{24} + v_{35} - (v_{13} v_{24} + v_{13} v_{35} + v_{24} v_{25} v_{35} + v_{24} v_{34} v_{35} + 1)
\end{dmath*}
As above, let $P,p$ denote the following polynomials in the remaining 5 $v$-variables: 
\begin{dgroup*}
\begin{dmath*}
P  = v_{13} v_{24} v_{25} v_{35} + v_{13} v_{24} v_{34} v_{35} + v_{13} + v_{24} v_{25} v_{34} v_{35} + v_{24} + v_{35}  
\end{dmath*}
\begin{dmath*}
p  = v_{13} v_{24} + v_{13} v_{35} + v_{24} v_{25} v_{35} + v_{24} v_{34} v_{35} + 1 
\end{dmath*}
\end{dgroup*}
Now, $P=Av_{34}+B$ and $p=av_{34} +b$ where 
\begin{dgroup*}
\begin{dmath*}
A  = v_{24} v_{35} (v_{13} + v_{25})
\end{dmath*}
\begin{dmath*}
a  = v_{24} v_{35} 
\end{dmath*}
\end{dgroup*}
and both $B$ and $b$ do not have the variable $v_{34}$.
Since $A-a \succeq 1$, by \ref{lemma:L1}
\begin{dmath*}
\pfaff Q  \succeq  A+B -a-b = v_{13} v_{24} v_{25} v_{35} + v_{13} v_{24} v_{35} + v_{13} + v_{24} + v_{35} - (v_{13} v_{24} + v_{13} v_{35} + v_{24} v_{35} + 1)
\end{dmath*}
Let $P,p$ denote the following polynomials in the remaining 4 $v$-variables: 
\begin{dgroup*}
\begin{dmath*}
P  = v_{13} v_{24} v_{25} v_{35} + v_{13} v_{24} v_{35} + v_{13} + v_{24} + v_{35}  
\end{dmath*}
\begin{dmath*}
p  = v_{13} v_{24} + v_{13} v_{35} + v_{24} v_{35} + 1 
\end{dmath*}
\end{dgroup*}
Now, $P=Av_{25}+B$ and $p=b$ where 
\begin{dgroup*}
\begin{dmath*}
A  = v_{13} v_{24} v_{35} 
\end{dmath*}
\end{dgroup*}
and both $B$ and $b$ do not have the variable $v_{25}$.
Since $A \succeq 1$, by \ref{lemma:L1}
\begin{dmath*}
\pfaff Q  \succeq  A+B -b = 2 v_{13} v_{24} v_{35} + v_{13} + v_{24} + v_{35} - (v_{13} v_{24} + v_{13} v_{35} + v_{24} v_{35} + 1)
\end{dmath*}
Let $P,p$ denote the following polynomials in the remaining 3 $v$-variables: 
\begin{dgroup*}
\begin{dmath*}
P  = 2 v_{13} v_{24} v_{35} + v_{13} + v_{24} + v_{35}  
\end{dmath*}
\begin{dmath*}
p  = v_{13} v_{24} + v_{13} v_{35} + v_{24} v_{35} + 1 
\end{dmath*}
\end{dgroup*}
Now, $P=Av_{24}+B$ and $p=av_{24} +b$ where 
\begin{dgroup*}
\begin{dmath*}
A  =  2 v_{13} v_{35} + 1
\end{dmath*}
\begin{dmath*}
a  =  v_{13} + v_{35}
\end{dmath*}
\end{dgroup*}
and both $B$ and $b$ do not have the variable $v_{24}$.
Since $A-a \succeq 1$, by \ref{lemma:L1}
\begin{dmath*}
\pfaff Q  \succeq  A+B -a-b = v_{13} v_{35}  \succeq 1,
\end{dmath*}
which concludes the proof. 
\end{remark}

\begin{remark}
The above procedure has been implemented in \verb+python+: the source code of
the script and a few 
lines of documentation can be found on the following github URL: 
\verb+https://github.com/dlfer/pfaffianinequaliites.git+. 
There are approximately 500 lines of code, structured in 28 functions, with a minimal number of comments
meant to try to help to understand its logic. The unusual design choice of not using symbolic polynomial 
libraries, is due 
to the need of optimizing the algorithm without using too much memory: 
hence instead of using polynomials or monomials from external libraries, 
the polynomials are encoded as python lists of \verb+Counters+, which are much simpler and faster to manipulate.   
\end{remark}

\section{Results for Collinear Central Configurations}
\label{section:5}

Now we can apply the previous results to the Pfaffian occurring in the 
inverse problem for collinear central configurations. 

\begin{theo}
\label{theo:pfaffQ}
Given any $\vq\in \conf{n}{\RR}$ such that $q_1>q_2>\ldots>q_n$. 
Then the corresponding matrix $Q$  
(with entries defined by \eqref{eq:defQ}, depending on the potential \ref{eq:potentialalpha},
as 
$Q_{ij} = \alpha (q_i - q_j)^{-1-\alpha}$ for $i<j$)
is skew-symmetric and satisfies 
\aref{positive}, 
\aref{quadratic} 
and \aref{monotone} of \ref{defi:T}:  thus $Q\in \matrices_n$. 
\end{theo}
\begin{proof}
The matrix $Q$ is clearly antisymmetric. 
If $i<j$ then $q_i > q_j$ and hence
$Q_{ij} = \alpha (q_i - q_j)^{-1-\alpha}>0$, hence \aref{positive} holds. 

For any $1< i<j < n$ the following hold:
\begin{align*}
Q_{i,j} Q_{i-1,j+1}  & = \alpha^2 (q_i-q_j)^{-1-\alpha} (q_{i-1} - q_{j+1})^{-1-\alpha} \\
Q_{i-1,j} Q_{i,j+1} & =  \alpha^2 (q_{i-1} - q_j)^{-1-\alpha} (q_i - q_{j+1})^{-1-\alpha} \\
\therefore Q_{i,j} Q_{i-1,j+1} & \geq   Q_{i-1,j} Q_{i,j+1} \\
\iff 
  (q_i-q_j)^{-1-\alpha} (q_{i-1} - q_{j+1})^{-1-\alpha} 
 & \geq   (q_{i-1} - q_j)^{-1-\alpha} (q_i - q_{j+1})^{-1-\alpha} \\
 \iff
  (q_i-q_j)(q_{i-1} - q_{j+1})
 & \leq   (q_{i-1} - q_j) (q_i - q_{j+1}) \\
\left( a := q_{i-1}, b := q_i, c := q_j, d := q_{j+1}  \right.
 &  \left. \implies   a>b>c>d \right) \\
 \iff
 (b-c)(a-d) \leq (a - c)(b-d)  & 
 \iff 
 ba - bd - ca + cd \leq ab - ad - cb + cd \\
 \iff 
  - bd - ca  \leq  - ad - cb   & \iff 
  ad - bd \leq ac - bc\\
  \iff d(a-b) \leq c(a-b) & \iff d \leq c  
\Longleftarrow d < c,  
\end{align*}
and hence \aref{quadratic} holds. 

Finally, 
\begin{align*}
(q_1 - q_2) & \leq \ldots 
\leq (q_1 - q_j) \leq \ldots \leq (q_1-q_{n-1}) \\
\implies (q_1 - q_2)^{-1-\alpha} &  \geq \ldots 
\geq (q_1 - q_j)^{-1-\alpha} \geq \ldots \geq (q_1-q_{n-1})^{-1-\alpha} \\
\end{align*}
and similarly
\begin{align*}
(q_2-q_n) & \geq \ldots \geq (q_j - q_n) \geq  \ldots \geq 
(q_{n-1} - q_n) \\
\implies 
(q_2-q_n)^{-1-\alpha} & \leq \ldots \leq (q_j - q_n)^{-1-\alpha} \leq  \ldots \leq 
(q_{n-1} - q_n)^{-1-\alpha}, \\
\end{align*}
which implies that \aref{monotone} holds.
Thus, all the properties requested in \ref{defi:T} hold true. 
\end{proof}

More generally, 
assume that the entries of  a skew-symmetric matrix $Q$ can be
written as in \eqref{eq:defQ} as $Q_{ij} = V(q_i - q_j)$ for all $i<j$, for a suitable 
odd function $V\from (0,+\infty) \to (0,+\infty)$; for example, 
assume that the potential $U$ is written in terms of a function $\varphi$
(cf. \eqref{eq:potentialalpha}) whose 
derivative (up to sign) in
\eqref{eq:defQ} yields such function $V$.
The $\alpha$-homogeneous case corresponds to $V(t) = \alpha \dfrac{t}{\abs{t}^{\alpha+2}}$. 
Assume that $V$ 
satisfies  the following two
properties:
\begin{axioms}{V}
\item\label{V:decreasing} $V$ is decreasing: $x< y \implies V(x) \geq V(y)$. 
\item\label{V:logconvex} $V$ is log-convex: $ \log V(x) $ is a convex function
of $x$. 
\end{axioms}
Then, the following theorem, which generalizes \ref{theo:pfaffQ},  holds.
\begin{theo}
\label{theo:2}
Assume that the function
  $V$ is decreasing \aref{V:decreasing} and log-convex \aref{V:logconvex}. 
Given $n$ real numbers $q_j$, $j=1,\ldots, n$, such that $q_1 > q_2 > \ldots
q_n$, and given
the anti-symmetric matrix $Q$ with entries
$Q_{ij}$ defined by $Q_{ij} = V(q_i - q_j )$ for all $1\leq i <  j \leq n$,
then 
$Q\in \matrices_n$. 
\end{theo}
\begin{proof}
By definition and \aref{V:decreasing}, properties \aref{positive} and \aref{monotone}
of \ref{defi:T} hold true.  It is left to prove \aref{quadratic}. 
Fix two indices $i = 2,\ldots, n-2$, and $ j = i+1 \ldots n-1$:
\begin{align*}  
Q_{i,j}Q_{i-1,j+1} & \geq  Q_{i-1,j} Q_{i,j+1}   \\
\iff V(q_i - q_j) V(q_{i-1} - q_{j+1} )  & \geq  V ( q_{i-1}- q_{j} ) V(q_i- q_{j+1} ) \\
\iff \log V(q_i - q_j) + \log V(q_{i-1} - q_{j+1} )  & \geq  \log V ( q_{i-1}- q_{j} ) + \log  V(q_i- q_{j+1} ) \\
\end{align*}
and the last inequality follows from \aref{V:logconvex}:
after defining
$a  := q_{i-1} - q_i$, 
$b  := q_{j} - q_{j+1} $, and 
$q_{ij} :=  q_i - q_j $, one has
\begin{align*}
q_{i-1} - q_{j} & = q_{i-1} - q_i  + q_{i} - q_j = a + q_{ij} \\
&  = \dfrac{a+b}{a+b}a + \dfrac{a+b}{a+b}q_{ij} 
 =  (a+b) \dfrac{a}{a+b} + \dfrac{a}{a+b} q_{ij} + \dfrac{b}{a+b} q_{ij} \\
& = \dfrac{a}{a+b} (a+b+q_{ij}) + \dfrac{b}{a+b} q_{ij} \\ 
& = 
 \dfrac{a}{a+b} (q_{i-1} - q_{j+1} )
+ \dfrac{ b }{a+b}  (q_i - q_j),
\end{align*}
and similarly 
\begin{align*} 
q_i - q_{j+1} & = q_i - q_j + q_j - q_{j+1} = q_{ij} + b \\
& = \frac{a+b}{a+b} q_{ij} + \frac{a+b}{a+b} b 
 = \frac{a}{a+b} q_{ij} + \frac{b}{a+b} q_{ij} + \frac{b}{a+b} (a+b) \\
& = \frac{a}{a+b} q_{ij} + \frac{b}{a+b} (q_{ij} + a + b ) \\
& = \frac{a}{a+b}  (q_i - q_j) 
+  \frac{b}{a+b} (q_{i-1} - q_{j+1} );
\end{align*}
and hence by \aref{V:logconvex} 
\begin{align*}
\log V (q_{i-1} - q_{j} ) & \leq 
\dfrac{a}{a+b} \log V (q_{i-1} - q_{j+1} )
+ \dfrac{b}{a+b}  \log V (q_i - q_j),\\
\log V ( q_i - q_{j+1} )  & \leq  
\dfrac{b}{a+b} \log V (q_{i-1} - q_{j+1} ) 
+  \dfrac{a}{a+b} \log V (q_i - q_j ),
\end{align*} 
which implies the claim  by adding the two last inequalities. 
\end{proof}

\begin{coro}
\label{coro:main}
For $n$ even, $n\leq 14$.
Assume that the potential $U$ is written as 
$U(\vq)  
= 
\sum_{1\leq i \leq j \leq n} {m_i m_j} \varphi(q_i - q_j)$,
where  $\varphi\from \RR \to \RR$ is a given even function
such that $V(t) = -\varphi'(t)$  satisfies \aref{V:decreasing} 
(decreasing) and \aref{V:logconvex} (log-convex). 
Then, if $\vq$ is a configuration such that $q_1 > q_2> \ldots > q_n$, 
the 
matrix $Q$ 
(with entries defined by \eqref{eq:defQ} as $Q_{ij} = V(q_i-q_j)$)
satisfies 
the following inequality
\begin{equation}
\label{pfaffianineq}
\pfaff Q 
 > 0. 
\end{equation}
\end{coro}
\begin{proof}
By \ref{theo:pfaffQ}, 
the matrix $Q$ belongs to $\matrices_n$. 
Hence by Corollary \ref{coro:mainineq} 
inequality \ref{eq:mainineq} holds, 
which can be written as 
\begin{equation}
\label{eq:ineqinter}
\pfaff Q \geq 
\left( \dfrac{Q_{1,n-1}}{Q_{1n}Q_{n-1,n}} \right)^{(n-2)/2} 
\left( \prod_{j=1,\ldots, n-1} Q_{jn}  \right) > 0 
\end{equation}
where $Q_{ij} = V(q_i - q_j)$ for $i<j$, and therefore the claim.
\end{proof}

\begin{coro}
\label{coro:alphahomo}
For even $n\leq 14$ and $\alpha>0$, if $U$ is a homogeneous potential of type \eqref{eq:potentialalpha},
$U(\vq) =  \sum_{1\leq i \leq j \leq n} \frac{m_i m_j}{\abs{q_i - q_j}^\alpha}$,
and $q_1>q_2>\ldots > q_n$, 
then the Pfaffian $\pfaff Q$ of the matrix  $Q$ 
defined as in \eqref{eq:defQ} by
$Q_{ij} =
\alpha \frac{q_i - q_j}{\abs{q_i - q_{j}}^{\alpha+2} }$
satisfies the inequality 
\begin{equation}
\label{pfaffianineqalpha}
\pfaff Q \geq 
\alpha^{n/2}
\left[
\left( 
\dfrac{1}{q_{n-1} - q_n} - \dfrac{1}{q_1 - q_n} 
\right)^{(n-2)/2} 
\left( \prod_{j=1,\ldots, n-1} 
(q_j - q_n)
\right)
\right]^{-1-\alpha} > 0. 
\end{equation} 
\end{coro}
\begin{proof}
Because of Theorem \ref{theo:pfaffQ}, inequality \ref{coro:mainineq}
(which is \ref{eq:ineqinter}) holds. 
Note that since $q_1>q_2>\ldots >q_n$, 
\begin{align*}
\dfrac{Q_{1,n-1}}{Q_{1n}Q_{n-1,n}} &  = \alpha^{-1} 
\left( 
\dfrac{ q_1 - q_{n-1} }{ (q_1 - q_n)(q_{n-1} - q_n) }
\right)^{-\alpha-1}  \\
& = 
\alpha^{-1} \left(
\dfrac{1}{q_{n-1} - q_n} - \dfrac{1}{q_1 - q_n}
\right)^{-\alpha - 1} 
\end{align*} 
and the claim follows. 
\end{proof}

\begin{remark}
It is clear that the homogeneous interaction function $\varphi(t) = \abs{t}^{-\alpha}$ yield a log-convex and decreasing 
function $V = -\varphi'(t)$. 
On the other hand, a finite sum of potentials with different homogeneities is not in general homogeneous, 
while it may yield a log-convex and decreasing function $V$. 
For $n =  6$ (and $8 \leq n\leq 14$ computer-assisted)
it is thus proven that 
$\pfaff Q>0$ for a potential $U$ which is not necessarily homogeneous.  
\end{remark}

\section{Declarations}
No funding was received to assist with the preparation of this manuscript.
The author has no financial or proprietary interests in any material discussed
in this article.

\printbibliography

@article{AlbouyInverseProblemCollinear2000,
  author =        {Albouy, Alain and Moeckel, Richard},
  journal =       {Celestial Mechanics and Dynamical Astronomy},
  month =         sep,
  number =        {2},
  pages =         {77--91},
  title =         {The {{Inverse Problem}} for {{Collinear Central
                   Configurations}}},
  volume =        {77},
  year =          {2000},
  doi =           {10.1023/A:1008345830461},
  issn =          {0923-2958, 1572-9478},
}

@article{Buchanancertaindeterminantsconnected1909,
  author =        {Buchanan, H. E.},
  journal =       {Bulletin of the American Mathematical Society},
  month =         feb,
  number =        {5},
  pages =         {227--232},
  title =         {On Certain Determinants Connected with a Problem in
                   Celestial Mechanics},
  volume =        {15},
  year =          {1909},
  doi =           {10.1090/S0002-9904-1909-01744-6},
  issn =          {0002-9904},
}

@article{xieAnalyticalProofCertain2014a,
  author =        {Xie, Zhifu},
  journal =       {Celestial Mechanics and Dynamical Astronomy},
  month =         jan,
  number =        {1},
  pages =         {89--97},
  title =         {An Analytical Proof on Certain Determinants Connected
                   with the Collinear Central Configurations in the
                   $n$-Body Problem},
  volume =        {118},
  year =          {2014},
  doi =           {10.1007/s10569-013-9525-4},
  issn =          {1572-9478},
}

@article{xieRemarksInverseProblem2022b,
  author =        {Xie, Zhifu},
  journal =       {Electronic Research Archive},
  number =        {7},
  pages =         {2540--2549},
  title =         {Remarks on the Inverse Problem of the Collinear
                   Central Configurations in the $ {{N}}$-Body
                   Problem},
  volume =        {30},
  year =          {2022},
  doi =           {10.3934/era.2022130},
  issn =          {2688-1594},
}

@article{ferrarioPfaffiansInverseProblem2020c,
  author =        {Ferrario, D. L.},
  journal =       {Celestial Mechanics and Dynamical Astronomy},
  month =         jun,
  number =        {6-7},
  pages =         {32},
  title =         {Pfaffians and the Inverse Problem for Collinear
                   Central Configurations},
  volume =        {132},
  year =          {2020},
  doi =           {10.1007/s10569-020-09975-3},
  issn =          {0923-2958, 1572-9478},
}

@article{ferrarioMultivaluedFixedPoints2020a,
  author =        {Ferrario, D. L.},
  journal =       {Rendiconti del Seminario Matematico. Universit\`a e
                   Politecnico Torino},
  number =        {2},
  pages =         {77--88},
  title =         {Multi-Valued Fixed Points and the Inverse Problem for
                   Central Configurations},
  volume =        {78},
  year =          {2020},
  issn =          {0373-1243,2704-999X},
}

@article{Moultonstraightlinesolutions1910,
  author =        {Moulton, F. R.},
  journal =       {Annals of Mathematics. Second Series},
  number =        {1},
  pages =         {1--17},
  title =         {The Straight Line Solutions of the Problem of $n$
                   Bodies},
  volume =        {12},
  year =          {1910},
  issn =          {0003-486X},
}

@article{knuthOverlappingPfaffians1996,
  author =        {Knuth, Donald E.},
  journal =       {Electronic Journal of Combinatorics},
  number =        {2},
  pages =         {Research Paper 5, approx. 13},
  title =         {Overlapping {{Pfaffians}}},
  volume =        {3},
  year =          {1996},
  issn =          {1077-8926},
}

@article{cayley1849determinants,
  author =        {Cayley, Arthur},
  journal =       {J. Reine Angew. Math.},
  pages =         {93},
  title =         {Sur Les D\'eterminants Gauches},
  volume =        {38},
  year =          {1849},
}

\end{document}